\documentclass{amsart}
\newtheorem{theorem}{Theorem}[section]
\newtheorem{lemma}[theorem]{Lemma}
\newtheorem{corollary}[theorem]{Corollary}
\newtheorem{proposition}[theorem]{Proposition}
\theoremstyle{definition}

\newenvironment{example}[1][Example]{\begin{trivlist}
\item[\hskip \labelsep {\bfseries #1}]}{\end{trivlist}}

\newenvironment{remark}[1][Remark]{\begin{trivlist}
\item[\hskip \labelsep {\bfseries #1}]}{\end{trivlist}}

\numberwithin{equation}{section}

\newcommand{\abs}[1]{\lvert#1\rvert}

\usepackage{amsmath}
\usepackage{cite}
\usepackage{units}
\usepackage[scr=rsfso]{mathalfa}
\newcommand{\LambertL}{\mathscr{L}}

\begin{document}

\title{Rapidly converging formulae for $\zeta(4k\pm 1)$}

\author{Shubho Banerjee}
\address{Department of Physics, Rhodes College, 2000 N. Parkway, Memphis, TN 38112}
\email{banerjees@rhodes.edu}
\email{bwilker849@gmail.com}
\thanks{The first author was supported by the Van Vleet Physics Professorship.}

\author{Blake Wilkerson}
\thanks{The second author was supported by the Mac Armour Physics Fellowship.}

\begin{abstract}
\emph{Version 2. This version corrects several typographical errors in version~1 (arXiv:1803.03291v1, 9 March 2018), states Theorem~\ref{TheoremZeta4k+1root3} only for $k$ not divisible by $3$, and updates the references.}\par\medskip

We provide rapidly converging formulae for the Riemann zeta function at odd integers using the Lambert series $\LambertL_q(s) = \sum_{n=1}^\infty n^{s} q^{n}/(1-q^n)$, $s=-(4k\pm 1)$. Our main formula for $\zeta(4k-1)$ converges at a rate of about $e^{-\sqrt{15}\pi}$ per term, and our main formula for $\zeta(4k+1)$ at a rate of about $e^{-4\pi}$ per term. For example, the first order approximation yields $\zeta(3)\approx\frac{\pi ^3 \sqrt{15}}{100} +e^{-\sqrt{15} \pi }\left[\frac{9}{4}+\frac{4}{\sqrt{15}}\sinh (\frac{\sqrt{15} \pi }{2})\right]$ which has an error of only order $10^{-10}$.
\end{abstract}


\maketitle

\section{Introduction}\label{Introduction}
In a recent paper\,\cite{BW} we analyzed the $\abs{q} \rightarrow 1^-$ asymptotic behavior of a Lambert series of the type
\begin{equation} \label{Lambert}
\hspace{1.02 in} \LambertL_q(s,x) = \sum_{n=1}^\infty \frac{n^s q^{n x}}{1-q^n},  \hspace{0.5 in} s \in \mathbb{C},
\end{equation}
with $\abs{q} < 1$ and $x>0$.
At $x=1$, the Lambert series $\LambertL_q(s,x)$ is the generating function for the divisor function $\sigma_s(n)$, the sum of the $s\text{th}$ powers of the divisors of $n$\,\cite[p.~827]{Abramowitz}:
\begin{equation}\label{divisor}
\sum_{n=1}^\infty \sigma_s(n) \, q^n = \LambertL_q(s,1).
\end{equation}
For odd positive integer $s$ values, the Lambert series is related to the Eisenstein series
\begin{equation}\label{Eisenstein}
E_{2k}(q) = 1 - \frac{4k}{B_{2k}} \sum_{n=1}^\infty \sigma_{2k-1}(n) \, q^n,
\end{equation}
which are modular in nature\,\cite[Ch.~1]{Apostol}.

In this paper we focus on the Lambert series at negative odd integer values of $s=-(4k\pm1)$, $k=1,2,3,...$ and at $x=1$. This Lambert series is related to the Riemann zeta function at the corresponding positive integers, $\zeta(4k\pm 1)$, through Ramanujan's formula\,\cite[Ch.~14, Entry~21(i), pp.~275--276]{Berndt}. Our main results are stated in Sections \ref{Sectionzeta(4k+1)}, \ref{SectionMoreZeta(4k+1)}, and \ref{Sectionzeta(4k-1)} where we provide a series of rapidly converging formulae for $\zeta(4k\pm1)$. Where comparison is possible, our results agree with those obtained using powerful computational methods \cite{Cohen,Plouffe1,Plouffe2}; see also \cite{Vepstas} for an analytic treatment of Plouffe's identities. The general formulae of Theorems~\ref{TheoremZeta4k+1p=3} and~\ref{TheoremZeta4k+1p=5}, together with the coefficients listed in the Appendix, were first posted in version~1 of this paper (arXiv:1803.03291, March 2018). The same formulae, with identical coefficients, later appeared as Families~C and~D of Broadhurst \cite{Broadhurst} (2024), where they were obtained by integer-relation searches; for example, his Family~C denominators $355, 11235, 2114515, 95520195,\ldots$ are the denominators of our Theorem~\ref{TheoremZeta4k+1p=3} (with $11235=7\times1605$).


\section{Lambert series}
\label{Lambert result}

In this section we focus on the Lambert series at negative odd integer arguments and $x=1$. Since $x$ is set to $1$ in this paper, to avoid redundancy we use the notation $\LambertL_q(s,1)\equiv\LambertL_q(s)$.

\begin{theorem}\label{TheoremOddLambert}
For odd negative integers $s$ the Lambert series at $x=1$ satisfies the following relations\emph{:}
\begin{enumerate}
\item For $s = -1$,
\begin{equation*}
\begin{aligned}
\LambertL_{e^{-2\pi t}}(-1)- \LambertL_{e^{-\frac{2\pi}{t}}}(-1)&=\frac{1}{2}\log t - \frac{\pi}{6} \sinh (\log t).
\end{aligned}
\end{equation*}

\item For $s = -(4k-1)=-3,-7,-11,\ldots$,
\begin{equation*}
\begin{aligned}
&\frac{1}{t^{2k-1}}\LambertL_{e^{-2\pi t}}(-4k+1)+t^{2k-1}\LambertL_{e^{-\frac{2\pi}{t}}}(-4k+1)\\
=&\,(2 \pi )^{4 k-1}\sum _{j=0}^{k} \frac{(-1)^{j+1} B_{2 j} B_{4 k-2 j} \cosh [(2 k-2 j) \log t]}{(2 j)!\, (4 k-2 j)!\,(1+\delta_{jk})}\\
 &\, - \zeta (4 k-1)\,\cosh [(2 k-1) \log t].\\
\end{aligned}
\end{equation*}

\item For $s = -(4k+1)=-5,-9,-13,\ldots$,
\begin{equation*}
\begin{aligned}
&\frac{1}{t^{2k}}\LambertL_{e^{-2\pi t}}(-4k-1)-t^{2k}\LambertL_{e^{-\frac{2\pi}{t}}}(-4k-1)\\
=&\,(2 \pi )^{4 k+1}\sum _{j=0}^{k} \frac{(-1)^{j+1} B_{2 j} B_{4 k+2-2 j} \sinh [(2 k+1-2 j) \log t]}{(2 j)!\, (4 k+2-2 j)!}\\
 &\,+ \zeta (4 k+1) \, \sinh (2 k\log t),\\
\end{aligned}
\end{equation*}

\end{enumerate}
\hspace{0.25 in} where $B_k$ is the $k\emph{\text{th}}$ Bernoulli number and $\delta_{jk}$ is the Kronecker delta.
\end{theorem}

\begin{remark}
Although the theorem is stated for negative odd integer arguments of the Lambert series, it holds true for positive odd arguments as well. The results for positive odd arguments, obtained by using negative $k$ values in the theorem, reproduce modular properties of the Eisenstein series\,(\ref{Eisenstein}).
\end{remark}

\begin{proof}
For the $s=-1$ case, by taking the logarithm of the complete expansion of the q-Pochhammer symbol provided in the Referee Remark 3.3 of Ref.\,\cite{BW} we get
\begin{equation}\label{Euler}
\begin{aligned}
\LambertL_{q}(-1)&=\frac{\pi^2}{6 t}+\frac{1}{2}\log \frac{t}{2\pi} -\frac{t}{24}+\sum_{n=1}^\infty \textrm{Li}_1 \big(e^{-\frac{4n\pi^2}{t}}\big)\\
&=\frac{\pi^2}{6 t}+\frac{1}{2}\log \frac{t}{2\pi} -\frac{t}{24}+\LambertL_{\!e^{-\frac{4\pi^2}{t}}}(-1),\\
\end{aligned}
\end{equation}
where $q=e^{-t}$ and $\textrm{Li}_s(q)$ is the polylogarithm function. Replacing $t$ with $2\pi t$ and rearranging the terms completes the proof.

For $s=-4k+1$ we begin with the result in Corollary 2.4(2) of Ref.\,\cite{BW}. Replacing $t$ with $2\pi t$, writing $\zeta^{\hspace{0.01in}\prime}(2-4k)$ in terms of $\zeta(4k-1)$, and writing the zeta function at even arguments in terms of Bernoulli numbers gives
\begin{equation*}
\begin{aligned}
&\frac{1}{t^{2k-1}}\LambertL_{e^{-2\pi t}}(-4k+1)\\
\approx\,&(2 \pi )^{4 k-1}\left[\frac{(-1)^{k+1} B_{2 k}^2}{(2 k)!^2 \, 2 }+\sum _{j=0}^{k-1} \frac{(-1)^{j+1} B_{2 j} B_{4 k-2 j} \cosh [(2 k-2 j) \log t]}{(2 j)!\, (4 k-2 j)!}\right]\\
 &\, -\zeta (4 k-1)\,\cosh [(2 k-1) \log t]. \\
\end{aligned}
\end{equation*}
The right hand side of the equation above is completely symmetric with respect to a $t \rightarrow 1/t$ transformation. Thus adding $t^{2k-1}\LambertL_{e^{-\frac{2\pi}{t}}}(-4k+1)$ to the left hand side of the equation to make it completely symmetric yields the result in the theorem. Note that such symmetry arguments only guarantee numerical accuracy. However, a comparison with Entry 21(i) in \cite[pp.~275--276]{Berndt} (see also \cite{Berndt1977}) shows that the result is exact, and there are no other error terms. In the final result the $B_{2k}^2$ term is included in the sum over $j$ by using $\delta_{jk}$ and extending the limit of the sum from $k\!-\!1$ to $k$.

Similarly, for the case $s=-4k-1$, Corollary 2.4(2) of Ref.\,\cite{BW} gives
\begin{equation*}
\begin{aligned}
&\frac{1}{t^{2k}}\LambertL_{e^{-2\pi t}}(-4k-1)\\
\approx &\,(2 \pi )^{4 k+1}\sum _{j=0}^{k} \frac{(-1)^{j+1} B_{2 j} B_{4 k+2-2 j} \sinh [(2 k+1-2 j) \log t]}{(2 j)!\, (4 k+2-2 j)!}\\
 &\,+ \zeta (4 k+1) \, \sinh (2 k\log t).\\
\end{aligned}
\end{equation*}
In this case the right hand side of the equation is antisymmetric with respect to a $t \rightarrow 1/t$ transformation. Thus, we add $t^{2k}\LambertL_{e^{-\frac{2\pi}{t}}}(-4k-1)$ to the left hand side of the equation to make the equation completely antisymmetric. As in the case above, a comparison with Entry 21(i) in \cite[pp.~275--276]{Berndt} shows that the result is exact, which yields the result in the theorem.

\end{proof}

\begin{corollary}\label{CorollaryZeta(4k-1)}
For positive odd integers $(4k-1)=3,7,11,...$ we have

\begin{equation*}
\begin{aligned}
\zeta(4k-1)&=(2 \pi)^{4 k-1}\! \sum _{j=0}^{k} \frac{(-1)^{j+1} B_{2 j}\, B_{4 k-2 j}}{(2 j)! \, (4 k-2 j)!\,(1+\delta _{jk})}   -2\LambertL_{e^{-2\pi}}(-4k+1).\\
\end{aligned}
\end{equation*}
\end{corollary}
\begin{proof}
Substituting $t=1$ in Theorem \ref{TheoremOddLambert}(2) and solving for $\zeta(4k-1)$ completes the proof.
\end{proof}

\begin{corollary}\label{CorollaryZeta(4k+1)}
For positive odd integers $(4k+1)=5,9,13,...$ we have

\begin{equation*}
\begin{aligned}
\zeta (4 k+1)&=(2 \pi )^{4 k+1}\sum _{j=0}^k \frac{(-1)^j (2 k+1-2 j)B_{2 j} B_{4 k+2-2 j}}{2 k \,(2 j)! \,(4 k+2-2 j)!}\\
&~~~-\frac{2 \pi e^{-2 \pi }   \LambertL'_{\!e^{-2 \pi }}\!\left(-4 k-1\right)}{k}-2  \LambertL_{\!e^{-2 \pi }}\!\left(-4 k-1\right),
\end{aligned}
\end{equation*}
where $\LambertL' = d\LambertL/dq$.
\end{corollary}
\begin{proof}
Differentiating Theorem \ref{TheoremOddLambert}(3) with respect to $t$ at $t=1$ and solving for $\zeta(4k+1)$ completes the proof.
\end{proof}

\begin{remark}
The results in Corollaries \ref{CorollaryZeta(4k-1)} and \ref{CorollaryZeta(4k+1)} are well known\,\cite{Berndt,Cohen,Sondow} but have been reproduced here for the sake of completeness. The results hold true for negative values of $k$ as well.
\end{remark}

\section{Formulae for $\zeta(4k+1)$}\label{Sectionzeta(4k+1)}
In Corollary\,\ref{CorollaryZeta(4k+1)} we established a formula for $\zeta(4k+1)$ that involved the Lambert series and its derivative with respect to $q$ at $q=e^{-2\pi}$. In this section we establish three new formulae for $\zeta(4k+1)$ that are in terms of the Lambert series only, and not its derivative.

First, we prove a general identity for the Lambert series that is true for all $s$ values in the form of the following lemma.

\begin{lemma} \label{LemmaLambert p=p}
For any positive prime integer $p$ and $s\in \mathbb{C}$ the Lambert series at $x=1$ satisfies\emph{:}
\begin{equation*}
\sum_{n=0}^{p-1} \LambertL_{q^{\frac{1}{p}} e^{i\frac{2\pi n}{p} }}(s) = (p^{s+1} +p)\LambertL_{q}(s)- p^{s+1}\LambertL_{q^p}(s).
\end{equation*}
\end{lemma}

\begin{proof}
Writing the Lambert series in terms of the divisor function using (\ref{divisor}) gives
\begin{equation*}
\begin{aligned}
\frac{1}{p}\sum_{n=0}^{p-1} \LambertL_{q^{\frac{1}{p}} e^{i\frac{2\pi n}{p} }}(s) &= \frac{1}{p} \sum_{n=0}^{p-1} \sum_{k = 1}^\infty \sigma_s(k) q^{\frac{k}{p}} e^{i\frac{2\pi n k}{p} }\\
&= \frac{1}{p} \sum_{k = 1}^\infty \sigma_s(k) q^{\frac{k}{p}} \sum_{n=0}^{p-1} e^{i\frac{2\pi n k}{p} }\\
&= \sum_{l = 1}^\infty \sigma_s(l p)\, q^{l}\\
&= (p^s+1) \sum_{l = 1}^\infty \sigma_s(l)\, q^{l} -p^s \sum_{p |l}^\infty \sigma_s(\tfrac{l}{p})\, q^{l}\\
&= (p^s+1) \sum_{l = 1}^\infty \sigma_s(l)\, q^{l}-p^s \sum_{m=1}^\infty \sigma_s(m)\, q^{m p},\\
\end{aligned}
\end{equation*}
where  $p |l$ indicates  sum only over integers $l$ that are divisible by $p$. Rewriting the two infinite sums in terms of Lambert series using (\ref{divisor}) and multiplying by $p$ on both sides proves the lemma.
In the proof above we used the following property of the divisor function\,\cite[Ch.~6]{Apostol}
\begin{equation*}
\begin{aligned}
\sigma_s(l p)= (p^s+1)\sigma_s(l) -p^s \sigma_s(\tfrac{l}{p}),
\end{aligned}
\end{equation*}
where $\sigma_s(\tfrac{l}{p})$ is non-zero only when $p$ divides $l$.
\end{proof}

\begin{remark}
Lemma\,\ref{LemmaLambert p=p} is not new: it is the Hecke relation $\sigma_s(lp)+p^s\sigma_s(l/p)=(1+p^s)\sigma_s(l)$ written for the generating function\,(\ref{divisor}). For odd positive $s$ it states that the Eisenstein series are eigenforms of the Hecke operator $T_p$ with eigenvalue $\sigma_s(p)=1+p^s$\,\cite[Ch.~6]{Apostol}. We include the short proof because the lemma is used for all $s\in\mathbb{C}$.
\end{remark}

\begin{remark}
For $p\ne 2 $ the limits of the sum over $n$ in the lemma can be replaced with the symmetric limits ranging from $-(p-1)/2$ to $(p-1)/2$. It is this symmetric range that we use in later proofs when $p=3$ and $p=5$.
\end{remark}

\begin{theorem}\label{TheoremZeta4k+1p=2}
The Riemann zeta function at odd integers $4k+1$ other than unity can be written in terms of Lambert series at $q=e^{-2 \pi}$ and $e^{-4 \pi}$ as\emph{:}
\begin{equation*}
\begin{aligned}
\zeta(4k+1)&=(2\pi)^{4k+1}\sum _{j=0}^k \frac{(-1)^{j+1} b_{jk}\, B_{2 j} B_{4 k+2-2 j} }{(2j)!(4 k+2-2 j)!}\\
&\hspace{0.175in}-\frac{2a_k+4}{a_k}\LambertL_{e^{-2 \pi }}(-4 k-1)+\frac{4}{a_k}\LambertL_{e^{-4 \pi }}(-4 k-1),\\
\end{aligned}
\end{equation*}
where
\begin{equation*}
\begin{aligned}
a_k&=2^{4 k+1}-(-1)^k 2^{2 k}-1,\\
b_{jk}&=\frac{2^{2 j-1} \left[1+(1+i)^{4 k+1-2 j}\right]-2^{4 k+1-2 j}\left[1+(1+i)^{2 j-1}\right]}{a_k}.\\
\end{aligned}
\end{equation*}
\end{theorem}

\begin{proof}
Applying Theorem\,\ref{TheoremOddLambert}(3) at $t=\frac{1}{2}$ gives

\begin{equation}\label{t=2}
\begin{aligned}
&~2^{2 k} \LambertL_{e^{- \pi }}(-4 k-1)-\frac{1}{2^{2 k}} \LambertL_{e^{-4 \pi }}(-4 k-1)\\
=&\,-(2 \pi )^{4 k+1} \sum _{j=0}^k \frac{(-1)^{j+1} B_{2 j} B_{4 k+2-2 j} \sinh (\log 2^{2 k+1-2j})}{(2 j)! \, (4 k+2-2 j)!}\\
&\,-\zeta (4 k+1) \sinh (\log 2^{2 k}).
\end{aligned}
\end{equation}
Applying Theorem\,\ref{TheoremOddLambert}(3) at $t=\frac{1}{2}-\frac{i}{2}=1/(1+ i)$ gives
\begin{equation}\label{t=1+i}
\begin{aligned}
&~(1+ i)^{2 k}\LambertL_{-e^{- \pi }}(-4 k-1)-\frac{1}{(1+ i)^{2 k}} \LambertL_{e^{-2 \pi }}(-4 k-1)\\
=&-(2 \pi )^{4 k+1} \sum _{j=0}^k \frac{(-1)^{j+1} B_{2 j} B_{4 k+2-2 j} \sinh [\log (1+ i)^{2 k+1-2j}]}{(2 j)! \, (4 k+2-2 j)!}\\
&-\zeta (4 k+1) \sinh [\log (1+ i)^{2 k}].\\
\end{aligned}
\end{equation}

Applying Lemma\,\ref{LemmaLambert p=p} at $q=e^{-2\pi}$ with $p=2$ gives

\begin{equation}\label{p=2}
\begin{aligned}
&\LambertL_{-e^{- \pi }}(-4 k-1)+\LambertL_{e^{- \pi }}(-4 k-1) \\
= &\,(2^{-4k} +2)\LambertL_{e^{-2\pi}}(-4k-1)- 2^{-4k}\LambertL_{e^{-4\pi}}(-4k-1).
\end{aligned}
\end{equation}
Eliminating $\LambertL_{-e^{- \pi }}(-4 k-1)$ and $\LambertL_{e^{- \pi }}(-4 k-1)$ from Equations (\ref{t=2}), (\ref{t=1+i}), and (\ref{p=2}) and solving for $\zeta(4k+1)$ in terms of $\LambertL_{e^{- 2\pi }}(-4 k-1)$ and $\LambertL_{e^{- 4\pi }}(-4 k-1)$ completes the proof of the theorem.
\end{proof}

\begin{theorem}\label{TheoremZeta4k+1p=3}
The Riemann zeta function at odd integers $4k+1$ other than unity can be written in terms of Lambert series at $q=-e^{-3 \pi}$, $e^{-4 \pi}$, and $e^{-6 \pi}$ as\emph{:}
\begin{equation*}
\begin{aligned}
\zeta(4k+1)
&=\frac{(2\pi)^{4k+1}}{2b_k}\!\sum _{j=0}^k \frac{(-1)^{j+1} c_{jk}\, B_{2 j} B_{4 k+2-2 j} }{(2j)!(4 k+2-2 j)!}+\frac{(-1)^k 2^{2 k+1}}{b_k}\LambertL_{-e^{-3 \pi }}(-4 k-1)\\
&\hspace{0.175in}-\frac{a_k }{2^{4 k-1} b_k}\LambertL_{e^{-4 \pi }}(-4 k-1) +\frac{2^{2 k} }{2^{2 k-1} b_k}\LambertL_{e^{-6 \pi }}(-4 k-1)\\
\end{aligned}
\end{equation*}
where
\begin{equation*}
\begin{aligned}
a_k&=\frac{2^{4 k} \left(3^{4 k+1}+1\right)}{2^{4 k+1}-(-1)^k 2^{2 k}+1},\\
b_k&=\frac{3^{4 k+1}-1}{2}-(-1)^k \, 2^{2 k}-a_k\left[\frac{2^{4 k+1}-(-1)^k 2^{2k}-1}{2^{4 k+1}}\right],\\
c_{jk}&=3^{2 j-1} \sum _{n=-1}^1 (1+i n)^{4 k+1-2 j}- 3^{4 k+1-2 j}\sum _{n=-1}^1 (1+i n)^{2 j-1}\\
&\hspace{0.14in}-a_k\left[\frac{ 1+(1+i)^{4 k+1-2 j}}{2^{4 k+1-2 j}}-\frac{1+(1+i)^{2 j-1} }{2^{2 j-1}}\right].\\
\end{aligned}
\end{equation*}
\end{theorem}

\begin{remark}This formula converges at about 4.09 digits for each additional term in the Lambert series. The first few examples of this theorem are listed in the Appendix. If needed, the Lambert series at $q=-e^{-3 \pi}$ can be replaced by Lambert series at positive only $q$ values using Lemma\,\ref{LemmaLambert p=p} at $q=e^{-6 \pi}$ with $p=2$,
\begin{equation*}
\begin{aligned}
\LambertL_{-e^{-3 \pi }}(-4 k-1)=&-\LambertL_{e^{-3 \pi }}(-4 k-1) +(2^{-4k} +2)\LambertL_{e^{-6\pi}}(-4k-1)\\
&- 2^{-4k}\LambertL_{e^{-12\pi}}(-4k-1).\\
\end{aligned}
\end{equation*}
However, the result as stated in the theorem is more compact and has the same rate of convergence as the one with positive $q$ values. Similar replacements can be made in any of the later results that involve a negative $q$ value.
\end{remark}

\begin{proof}
The proof is a generalization of the proof of Theorem\,\ref{TheoremZeta4k+1p=2} to the $p=3$ case. Applying Theorem\,\ref{TheoremOddLambert}(3) at $t=\frac{1}{3}$ and $t=\frac{1}{3}\pm\frac{i}{3}$ yields three equations involving $\LambertL_{e^{-\frac{2\pi }{3}}}(-4 k-1) $, $\LambertL_{e^{-\frac{2\pi }{3}\pm\frac{2\pi i }{3}}}(-4 k-1)$, and $\zeta(4k+1)$ in terms of $\LambertL_{-e^{- 3\pi }}(-4 k-1)$, and $\LambertL_{e^{- 6\pi }}(-4 k-1)$. Applying Lemma\,\ref{LemmaLambert p=p} at $q=e^{-2 \pi}$ with $p=3$ gives
\begin{equation*}\label{p=3}
\begin{aligned}
\sum_{n=-1}^1 \LambertL_{e^{-\frac{2\pi}{3}+i\frac{2\pi n}{3} }}(-4k-1)= (3^{-4k} +3)\LambertL_{e^{-2\pi}}(-4k-1)- 3^{-4k}\LambertL_{e^{-6\pi}}(-4k-1).
\end{aligned}
\end{equation*}
Theorem\,\ref{TheoremZeta4k+1p=2} supplies the fifth equation that relates $\LambertL_{e^{- 2\pi }}(-4 k-1)$ to $\LambertL_{e^{- 4\pi }}(-4 k-1)$ and $\zeta(4k+1)$.

Using these five equations to eliminate $\LambertL_{e^{-\frac{2\pi }{3}}}(-4 k-1) $, $\LambertL_{e^{-\frac{2\pi }{3}\pm\frac{2\pi i }{3}}}(-4 k-1)$, and $\LambertL_{e^{- 2\pi }}(-4 k-1)$ and solving for $\zeta(4k+1)$ in terms of $\LambertL_{-e^{- 3\pi }}(-4 k-1)$, $\LambertL_{e^{- 4\pi }}(-4 k-1)$, and $\LambertL_{e^{- 6\pi }}(-4 k-1)$ completes the proof of the theorem.
\end{proof}

\begin{theorem}\label{TheoremZeta4k+1p=5}
The Riemann zeta function at odd integers $4k+1$ other than unity can be written in terms of Lambert series at $q=e^{-4 \pi}$, $-e^{-5 \pi},$ and $e^{-10 \pi}$ as\emph{:}
\begin{equation*}
\begin{aligned}
\zeta(4k+1)=&\frac{(2\pi)^{4k+1}}{2b_k}\sum _{j=0}^k \frac{(-1)^{j+1} c_{jk}\, B_{2 j} B_{4 k+2-2 j} }{(2j)!(4 k+2-2 j)!}-\frac{1}{2^{4 k-1} b_k}\LambertL_{e^{-4 \pi }}(-4 k-1)\\
&+\frac{(-1)^k 2^{2 k+1} a_k }{b_k}\LambertL_{-e^{-5 \pi }}(-4 k-1)+\frac{2a_k }{b_k}\LambertL_{e^{-10 \pi }}(-4 k-1),\\
\end{aligned}
\end{equation*}
where
\begin{equation*}
\begin{aligned}
a_k&=\frac{2^{4 k+1}-(-1)^k 2^{2 k}+1}{2^{4 k} \left[5^{4 k+1}-2\times 5^{2 k} \cos (4 k \tan^{-1}2)+1\right]},\\
b_k&=\frac{ a_k }{2} \bigg[5^{4 k+1}-\sum _{n=-2}^2 (1+i n)^{4 k}\bigg]-\frac{2^{4 k+1}-(-1)^k 2^{2 k}-1}{2^{4 k+1}},\\
c_{jk}&=a_k \bigg[5^{2 j-1}\! \sum _{n=-2}^2 (1+i n)^{4 k+1-2 j}-5^{4 k+1-2 j} \!\sum _{n=-2}^2 (1+i n)^{2 j-1}\bigg]\\
&\hspace{0.175in}-\bigg[\frac{1+(1+i)^{4 k+1-2 j}}{2^{4 k+1-2 j}}-\frac{1+(1+i)^{2 j-1}}{2^{2 j-1}}\bigg].\\
\end{aligned}
\end{equation*}
\end{theorem}

\begin{remark}This formula has a convergence rate of about 5.45 digits per term and is our best result for $\zeta(4k+1)$. The first few examples of this theorem are listed in the Appendix.
\end{remark}

\begin{proof}
The proof of this theorem is very similar to that of Theorem\,\ref{TheoremZeta4k+1p=3}.  Theorem\,\ref{TheoremOddLambert}(3) is now applied at $t=\frac{1}{5},\frac{1}{5}\pm\frac{i}{5}$, and $\frac{1}{5}\pm\frac{2i}{5}$ and Lemma\,\ref{LemmaLambert p=p} is applied at $q=e^{-2\pi}$ with $p=5$ instead of $p=3$. The rest of the argument is the same.
\end{proof}

\section{Additional formulae for $\zeta(4k+1)$}\label{SectionMoreZeta(4k+1)}
In the previous section we provided formulae for $\zeta(4k+1)$ in terms of Lambert series at $q$ values of the form $e^{-2^n\! p \pi}$ where $p$ was a prime equal to $2,3,$ or $5$  and $n$ was $0,1,$ or $2$. In this section we provide an additional class of formulae for $\zeta(4k+1)$ using Lambert series with $q$ values of the form $e^{-2^n \! \sqrt{m} \pi}$ where $m$ is equal to $3,7,$ or $15$ and $n$ is $0,1,$ or $2$.

First we state the following lemma, which combines two Lambert series to produce a series with a stronger convergence rate per term.

\begin{lemma} \label{LemmaLambert iq-(-iq)}
For any $s\in \mathbb{C}$ the Lambert series at $x=1$ satisfies\emph{:}
\begin{equation*}
\begin{aligned}
\mathcal{S}_{q}(s)\equiv i\left[\LambertL_{i\sqrt{q}}(s) -\LambertL_{-i\sqrt{q}}(s)\right] =\sum_{n=0}^\infty (-1)^{n+1}(2n+1)^s \textrm{\emph{sech}}\!\left[(n+\tfrac{1}{2})\log q\right].\\
\end{aligned}
\end{equation*}
\end{lemma}

\begin{proof}
\begin{equation*}
\begin{aligned}
\LambertL_{iq}(s) -\LambertL_{-iq}(s)&= \sum_{n=1}^\infty \frac{n^s (iq)^{n}}{1-(iq)^{n}}-\sum_{n=1}^\infty \frac{n^s (-iq)^{n}}{1-(-iq)^{n}}\\
&= \sum_{m=0}^\infty \frac{(2m+1)^s (iq)^{2m+1}}{1-(iq)^{2m+1}}-\sum_{m=0}^\infty \frac{(2m+1)^s (-iq)^{2m+1}}{1-(-iq)^{2m+1}}\\
&= \sum_{m=0}^\infty \frac{2(2m+1)^s(iq)^{2m+1}}{1-(iq)^{4m+2}}\\
&= -2i\sum_{m=0}^\infty \frac{(-1)^{m+1}(2m+1)^s q^{2m+1}}{1+q^{4m+2}}\\
&= -i\sum_{m=0}^\infty (-1)^{m+1}(2m+1)^s \textrm{sech}\!\left[(2m+1)\log q\right].\\
\end{aligned}
\end{equation*}
Replacing $q$ with $\sqrt{q}$ and multiplying by $i$ on both sides completes the proof.
\end{proof}

We now state three formulae for $\zeta(4k+1)$, in increasing order of their convergence rates.

\begin{theorem}\label{TheoremZeta4k+1root3}
For $k=1,2,3,\ldots$ not divisible by $3$, the Riemann zeta function at odd integers $4k+1$ can be written in terms of the Lambert series at $q=-e^{-\sqrt{3} \pi}$ as\emph{:}
\begin{equation*}
\begin{aligned}
\zeta(4k+1)=&\,(2\pi)^{4k+1}\sum _{j=0}^k \frac{(-1)^{j} b_{jk}\, B_{2 j} B_{4 k+2-2 j} }{(2j)!(4 k+2-2 j)!}-a_k\LambertL_{-e^{-\sqrt{3} \pi }}(-4 k-1)\\
\end{aligned}
\end{equation*}
where
\begin{equation*}
\begin{aligned}
a_k&=\frac{2^{4 k+1}+1}{2^{4 k}-\cos \big(\frac{2 \pi  k}{3}\big)},\\
b_{jk}&=\frac{2^{2 j-1} \sin\! \big[\frac{(2 k-1-j)\pi}{3}\big]+2^{4 k+1-2 j} \sin\! \big[\frac{(j+1)\pi}{3}\big]}{2^{4 k}-\cos \big(\frac{2 \pi  k}{3}\big)}.\\
\end{aligned}
\end{equation*}
\end{theorem}
\begin{proof}
Substituting $t=(\sqrt{3}\pm i)/2$ in Theorem\,\ref{TheoremOddLambert}(3) and subtracting the two results gives one equation in  $\LambertL_{- e^{-\sqrt{3} \pi }}(-4 k-1)$ and $\zeta(4k+1)$. Solving for $\zeta(4k+1)$ yields the final result (cf.\ Proposition~2.5 of \cite{Berndt1977}). In this equation $\zeta(4k+1)$ is multiplied by $\sin(k\pi/3)$, which vanishes when $3$ divides $k$; for those $k$ the substitution gives no information about $\zeta(4k+1)$. For all other $k$ one has $a_k=2$.
\end{proof}

\begin{theorem}\label{TheoremZeta4k+1root7}
The Riemann zeta function at odd integers $4k+1$ other than unity can be written in terms of the Lambert series at $q=-e^{-\sqrt{7}\pi}$ and $e^{-2\sqrt{7}\pi}$ as\emph{:}
\begin{equation*}
\begin{aligned}
\zeta(4k+1)=&\,(2\pi)^{4k+1}\sum _{j=0}^k \frac{(-1)^{j} c_{jk}\, B_{2 j} B_{4 k+2-2 j} }{(2j)!(4 k+2-2 j)!}-a_k\LambertL_{-e^{-\sqrt{7} \pi }}(-4 k-1)\\
&+b_k\LambertL_{e^{-2\sqrt{7} \pi }}(-4 k-1),\\
\end{aligned}
\end{equation*}
where
\begin{equation*}
\begin{aligned}
a_k&=\frac{2^{2 k+1}+2^{-2 k}-2 \cos (k 4 \theta  )}{2^{2 k}-\cos (4 k \theta)},\\
b_k&=\frac{2^{-2 k}}{2^{2 k}-\cos (4 \theta  k)},\\
c_{jk}&=\frac{2^{2k+\frac{1}{2}-j} \cos[ (2  j-1) \theta]-2^{j-\frac{1}{2}} \cos [ (4 k+1-2 j)\theta]}{2^{2 k}- \cos (4 k\theta)},\\
\end{aligned}
\end{equation*}
and $\theta=\cot^{-1}\sqrt{7}$.
\end{theorem}

\begin{proof}
Substituting $t=(\sqrt{7}\pm i)/4$ in Theorem\,\ref{TheoremOddLambert}(3) and adding gives one equation in $\LambertL_{\pm i e^{-{\sqrt{7} \pi}/{2} }}(-4 k-1)$, $\LambertL_{- e^{-\sqrt{7} \pi }}(-4 k-1)$, and $\zeta(4k+1)$. Eliminating the sum of $\LambertL_{\pm i e^{-{\sqrt{7} \pi}/{2} }}(-4 k-1)$ using Lemma\,\ref{LemmaLambert p=p} for $p=2$ at $q=-e^{-\sqrt{7} \pi }$ yields the final result.
\end{proof}

\begin{theorem}\label{TheoremZeta4k+1root15}
The Riemann zeta function at odd integers $4k+1$ other than unity can be written in terms of the Lambert series at $q=-e^{-\sqrt{15} \pi}$ and $ e^{-2\sqrt{15} \pi}$, and the series  $\mathcal{S}_{e^{-\sqrt{15} \pi }}(-4 k-1)$ as\emph{:}
\begin{equation*}
\begin{aligned}
\zeta(4k+1)=&\,(2\pi)^{4k+1}\sum _{j=0}^k \frac{(-1)^{j} c_{jk}\, B_{2 j} B_{4 k+2-2 j} }{(2j)!(4 k+2-2 j)!}+\cot (2 k \theta)\, \mathcal{S}_{e^{-\sqrt{15} \pi }}(-4 k-1)\\
& -\left(2^{-4 k}\!+2\right)\LambertL_{-e^{-\sqrt{15} \pi }}(-4 k-1)+2^{-4 k}\LambertL_{e^{-2\sqrt{15} \pi }}(-4 k-1)\\
\end{aligned}
\end{equation*}
where
\begin{equation*}
\begin{aligned}
c_{jk}&=\csc (2 k \theta) \sin [(2 k\!+\!1\!-\!2 j)\,\theta],\\
\end{aligned}
\end{equation*}
and $\theta=\cot^{-1}\sqrt{15}$.
\end{theorem}

\begin{remark}
This result converges almost as rapidly as Theorem\,\ref{TheoremZeta4k+1p=5} at 5.28 decimal digits per term. Some examples of this theorem are listed in the Appendix.
\end{remark}

\begin{proof}
Substituting $t=(\sqrt{15}\pm i)/4$ in Theorem\,\ref{TheoremOddLambert}(3) and subtracting the two results gives one equation containing both the sum and difference of $\LambertL_{\pm i e^{-{\sqrt{15}\pi}/{2} }}(-4 k-1)$,  and $\zeta(4k+1)$. Eliminating the sum of $\LambertL_{\pm i e^{-{\sqrt{15} \pi}/{2} }}(-4 k-1)$ using Lemma\,\ref{LemmaLambert p=p} for $p=2$ at $q=-e^{-\sqrt{15} \pi }$, and eliminating the difference of $\LambertL_{\pm i e^{-\sqrt{15} \pi /2 }}(-4 k-1)$ using Lemma\,\ref{LemmaLambert iq-(-iq)} at $q=e^{-\sqrt{15} \pi }$ yields the final result.
\end{proof}

\section{Formulae for $\zeta(4k-1)$}\label{Sectionzeta(4k-1)}
With minor changes, the methods of Section\,\ref{SectionMoreZeta(4k+1)} give similar results for $\zeta(4k-1)$, stated in the three theorems below.
\begin{theorem}\label{TheoremZeta4k-1root3}
The Riemann zeta function at odd integers $4k-1$ can be written in terms of Lambert series at $q=e^{-\sqrt{3}\pi}$, $e^{-2\sqrt{3}\pi}$, and $e^{-4\sqrt{3}\pi}$ as\emph{:}
\begin{equation*}
\begin{aligned}
\zeta(4k-1)=\,&(2\pi)^{4k-1}\sum _{j=0}^k \frac{(-1)^{j+1} b_{jk}\, B_{2 j} B_{4 k-2 j} }{(2j)!(4 k-2 j)!}-\frac{2^{4 k+1}+4}{a_k}\LambertL_{-e^{-\sqrt{3} \pi }}(-4 k+1)\\
&+\frac{4+8 \sin \big[\frac{(4 k -5)\pi}{6}\big]}{a_k}\LambertL_{e^{-2\sqrt{3} \pi }}(-4 k+1),\\
\end{aligned}
\end{equation*}
where
\begin{equation*}
\begin{aligned}
a_k&=2^{4 k}\!+4 \sin\! \big[\tfrac{(4k+1) \pi }{6}\big],\\
b_{jk}&=\frac{2^{4 k+1-2 j} \cos \!\big[\frac{  (2 j-1)\pi}{6}\big]+2^{2 j+1} \cos \big[\frac{(4 k-1-2 j)\pi}{6} \big]}{a_k (1+\delta _{jk})}.\\
\end{aligned}
\end{equation*}
\end{theorem}

\begin{proof}
Substituting $t=(\sqrt{3}\pm i)/4$ in Theorem\,\ref{TheoremOddLambert}(2) and adding the results gives one equation in  $\LambertL_{\pm i e^{-\sqrt{3} \pi/2 }}(-4 k+1)$, $\LambertL_{e^{-2\sqrt{3} \pi }}(-4 k+1)$, and $\zeta(4k-1)$. Eliminating $\LambertL_{\pm i e^{-\sqrt{3} \pi/2 }}(-4 k+1)$ in favor of $\LambertL_{-e^{-\sqrt{3} \pi }}(-4 k+1)$ and $\LambertL_{ e^{-2\sqrt{3} \pi }}(-4 k+1)$ using Lemma\,\ref{LemmaLambert p=p} for $p=2$ at $q=-e^{-\sqrt{3} \pi }$ yields the final result.
\end{proof}

\begin{theorem}\label{TheoremZeta4k-1root7}
The Riemann zeta function at odd integers $4k-1$ can be written in terms of Lambert series at $q=-e^{-\sqrt{7}\pi}$ and $ e^{-2\sqrt{7}\pi}$ as\emph{:}
\begin{equation*}
\begin{aligned}
\zeta(4k-1)=&\,(2\pi)^{4k-1}\sum _{j=0}^k \frac{(-1)^{j+1} c_{jk}\, B_{2 j} B_{4 k-2 j} }{(2j)!(4 k-2 j)!}-\frac{b_k}{a_k}\LambertL_{-e^{-\sqrt{7} \pi }}(-4 k+1)\\
&+\frac{2^{-2 k+2}}{a_k}\LambertL_{e^{-2\sqrt{7} \pi }}(-4 k+1),\\
\end{aligned}
\end{equation*}
where
\begin{equation*}
\begin{aligned}
a_k&=2^{2 k}+2 \cos [(4 k\!-\!2)\,\theta],\\
b_k&=2^{2 k+1}+2^{-2 k +2}+2^{\frac{3}{2}} \sin [(4 k\!+\!1)\,\theta]+2 \cos (4 k \theta ),\\
c_{jk}&=\frac{2^{2 k+\frac{1}{2}-j} \cos [(2 j\!-\!1)\,\theta]+2^{j+\frac{1}{2}} \cos [(4 k\!-\!1\!-\!2 j)\,\theta]}{a_k (1+\delta _{jk})},\\
\end{aligned}
\end{equation*}
and $\theta=\cot^{-1}\sqrt{7}$.
\end{theorem}

\begin{remark}
This formula for $\zeta(4k-1)$ has a convergence rate of about 3.6 decimal digits per term which is about 0.9 decimal digits per term faster than that of Corollary\,\ref{CorollaryZeta(4k-1)}. Some examples of this theorem are listed in the Appendix.
\end{remark}

\begin{proof}
Substituting $t=(\sqrt{7}\pm i)/4$ in Theorem\,\ref{TheoremOddLambert}(2) and adding the results gives one equation in $\LambertL_{\pm i e^{-{\sqrt{7} \pi}/{2} }}(-4 k+1)$, $\LambertL_{- e^{-\sqrt{7} \pi }}(-4 k+1)$, and $\zeta(4k-1)$. Eliminating the sum of $\LambertL_{\pm i e^{-{\sqrt{7} \pi}/{2} }}(-4 k+1)$ using Lemma\,\ref{LemmaLambert p=p} for $p=2$ at $q=-e^{-\sqrt{7} \pi }$ yields the final result.
\end{proof}

\begin{theorem}\label{TheoremZeta4k-1root15}
The Riemann zeta function at odd integers $4k-1$ can be written in terms of Lambert series at $q=-e^{-\sqrt{15} \pi}$ and $ e^{-2\sqrt{15} \pi}$, and the series $\mathcal{S}_{e^{-\sqrt{15} \pi }}(-4 k+1)$ as\emph{:}
\begin{equation*}
\begin{aligned}
\zeta(4k-1)=&\,(2\pi)^{4k-1}\sum _{j=0}^k \frac{(-1)^{j+1} c_{jk}\, B_{2 j} B_{4 k-2 j} }{(2j)!(4 k-2 j)!}+a_k \, \mathcal{S}_{e^{-\sqrt{15} \pi }}(-4 k+1)\\
& -\left(2^{-4 k+2}+2\right)b_k\LambertL_{-e^{-\sqrt{15} \pi }}(-4 k+1)\\
&+ 2^{-4 k+2} b_k\LambertL_{e^{-2\sqrt{15} \pi }}(-4 k+1),\\
\end{aligned}
\end{equation*}
where
\begin{equation*}
\begin{aligned}
a_k&=\frac{\cos [(4 k-1)\theta ]-2 \sin (4 k\theta)}{4 \cos^2[(2 k-1)\theta]},\\
b_k&=\frac{2+2 \cos (4k \theta)+\sin [(4k-1) \theta  ]}{4 \cos^2[(2 k-1)\theta]},\\
c_{jk}&=\frac{\cos [(2 k-2 j)\theta  ]}{(1+\delta _{jk})\cos [(2 k-1)\theta ]},\\
\end{aligned}
\end{equation*}
and $\theta=\cot^{-1}\sqrt{15}$.
\end{theorem}

\begin{remark}
This formula has a convergence rate of about 5.28 decimal digits per term which is the best of our results for $\zeta(4k-1)$. Some examples of this theorem are listed in the Appendix.
\end{remark}

\begin{proof}
Substituting $t=(\sqrt{15}\pm i)/4$ in Theorem\,\ref{TheoremOddLambert}(2) and adding the results gives one equation containing both the sum and difference of $\LambertL_{\pm i e^{-{\sqrt{15} \pi/{2}}}}(-4 k+1)$,  and $\zeta(4k-1)$. Eliminating the sum of $\LambertL_{\pm i e^{-{\sqrt{15} \pi}/{2} }}(-4 k+1)$ using Lemma\,\ref{LemmaLambert p=p} for $p=2$ at $q=-e^{-\sqrt{15} \pi }$, and eliminating the difference of $\LambertL_{\pm i e^{-\sqrt{15}\pi / 2 }}(-4 k+1)$ using Lemma\,\ref{LemmaLambert iq-(-iq)} at $q=e^{-\sqrt{15} \pi }$ yields the final result.
\end{proof}

\section{Additional results for the Lambert series}\label{SectionAdditionalLambertResults}

Other interesting results follow from the earlier sections. For example, Theorem\,\ref{TheoremOddLambert} gives
\begin{theorem}\label{TheoremIndependentOfZeta}
For any constant $a$ with $Re(a)>0$, the Lambert series $\LambertL_q(s)$ at negative odd integers $s$ satisfies the following\emph{:}
\begin{enumerate}
\item For $s = -(4k+1) = -1, -5, -9, \ldots$,
\begin{equation*}
\begin{aligned}
&\frac{1}{t^{2 k}}\Big[a^{2 k} \LambertL_{e^{-\frac{2\pi t}{a}}}(-4k-1)-(a^{2 k}+a^{-2 k})\LambertL_{e^{-2\pi t}}(-4k-1)\Big.\\
&\hspace{0.25in}\Big.+a^{-2 k}  \LambertL_{e^{-2\pi a t}}(-4k-1)        \Big]\\
-&t^{2 k}\Big[\Big.a^{2 k} \LambertL_{e^{-\frac{2\pi}{a t}}}(-4k-1)-(a^{2 k}+a^{-2 k})\LambertL_{e^{-\frac{2\pi}{ t}}}(-4k-1)\\
&\hspace{0.25in}\Big.+ a^{-2 k}  \LambertL_{e^{-\frac{2\pi a}{t}}}(-4k-1)         \Big]\\
=&\,(2 \pi )^{4 k+1}\sum _{j=0}^{k}\frac{(-1)^{j} b_{jk}(a) B_{2 j} B_{4 k+2-2 j} \sinh [(2 k+1-2 j) \log t]}{(2 j)!\, (4 k+2-2 j)!}.\\
\end{aligned}
\end{equation*}

\item For $s = -(4k-1) = -3, -7, -11, \ldots$,
\begin{equation*}
\begin{aligned}
&\frac{1}{t^{2 k-1}}\Big[a^{2 k-1} \LambertL_{e^{-\frac{2\pi t}{a}}}(-4k+1)-(a^{2 k-1}+a^{-2 k+1})\LambertL_{e^{-2\pi t}}(-4k+1)\Big.\\
&\hspace{.37in}\Big.+{a^{-2 k+1}}  \LambertL_{e^{-2\pi a t}}(-4k+1)        \Big]\\
+&\,t^{2 k-1}\Big[\Big.a^{2 k-1} \LambertL_{e^{-\frac{2\pi}{a t}}}(-4k+1)-(a^{2 k-1}+a^{-2 k+1})\LambertL_{e^{-\frac{2\pi}{ t}}}(-4k+1)\\
&\hspace{0.37in}\Big.+ a^{-2 k+1}  \LambertL_{e^{-\frac{2\pi a}{t}}}(-4k+1)         \Big]\\
=&\,(2 \pi )^{4 k-1}\sum _{j=0}^{k}\frac{(-1)^j c_{jk}(a) B_{2 j} B_{4 k-2 j} \cosh [(2 k-2 j) \log t]}{(2 j)!\, (4 k-2 j)!},\\
\end{aligned}
\end{equation*}

\end{enumerate}
\hspace{0.25 in} where
\begin{equation*}
b_{jk}(a)=a^{2 k}+a^{-2 k}-a^{2k+1-2j}-a^{-2k-1+2j},
\end{equation*}
\hspace{0.25 in} and
\begin{equation*}
c_{jk}(a)=\frac{a^{2 k-1}+a^{-2 k+1}-a^{2 k-2j}-a^{-2 k+2j}}{(1+\delta_{jk})}.
\end{equation*}
\end{theorem}

\begin{proof}
The proof follows from making the substitutions $t\rightarrow at$ and $t\rightarrow t/a$ in Theorem\,\ref{TheoremOddLambert} and then combining them with the original result in Theorem\,\ref{TheoremOddLambert} in the manner stated in the theorem above. The term containing the zeta function at odd arguments, $\zeta(4k\pm1)$, is eliminated. The special case of $s=-1$ is included as the $k=0$ case in $s=-(4k+1)$. The $\log t$ term is eliminated in this case.
\end{proof}

As examples of Theorem\,\ref{TheoremIndependentOfZeta}, odd powers of $\pi$ can be calculated in terms of the Lambert series at desired $q$ values as stated below.
\begin{example}\label{ExamplePi^(4k+1)}
By choosing $t=\frac{1}{2}\pm \frac{i}{2}$ and $a=\frac{1}{2}$ in Theorem\,\ref{TheoremIndependentOfZeta}(1), eliminating $\LambertL_{-e^{-\pi}}(-4k-1)$ using Lemma\,\ref{LemmaLambert p=p} at $q=e^{-2\pi}$ and $p=2$, and eliminating the sum $\LambertL_{i e^{-\frac{\pi}{2}}}(-4k-1)+\LambertL_{-i e^{-\frac{\pi}{2}}}(-4k-1)$ using Lemma\,\ref{LemmaLambert p=p} with $p=2$ (applied at $q=-e^{-\pi}$ and then at $q=e^{-2\pi}$) we get

\begin{equation*}
\begin{aligned}
\pi&=72 \, \LambertL_{e^{-\pi}}(-1)-96 \, \LambertL_{e^{-2\pi}}(-1)+24 \, \LambertL_{e^{-4\pi}}(-1),\\
\pi^5&=7056 \,\LambertL_{e^{-\pi}}(-5)-6993 \, \LambertL_{e^{-2\pi}}(-5)-63 \, \LambertL_{e^{-4\pi}}(-5),\\
\pi^9&= \frac{28226880}{41}\LambertL_{e^{-\pi}}(-9)-\frac{112920885}{164} \LambertL_{e^{-2\pi}}(-9)+\frac{13365}{164} \LambertL_{e^{-4\pi}}(-9).\\
\end{aligned}
\end{equation*}

\end{example}

\begin{example}\label{ExamplePi^(4k-1)}
By choosing $t=1$ and $a=\frac{1}{2}$ in Theorem\,\ref{TheoremIndependentOfZeta}(2) we get

\begin{equation*}
\begin{aligned}
\pi^3&=720 \, \LambertL_{e^{-\pi}}(-3)-900  \, \LambertL_{e^{-2\pi}}(-3)+180  \, \LambertL_{e^{-4\pi}}(-3),\\
\pi^7&=\frac{907200}{13} \LambertL_{e^{-\pi}}(-7)-70875 \,\LambertL_{e^{-2\pi}}(-7)+\frac{14175}{13}\LambertL_{e^{-4\pi}}(-7),\\
\pi^{11}\!\!&= \!\frac{27243216000}{4009}\LambertL_{e^{-\pi}}\!(-11)-\!\frac{218158565625}{32072} \LambertL_{e^{-2\pi}}\!(-11)\!+\!\frac{212837625}{32072} \LambertL_{e^{-4\pi}}\!(-11).\\
\end{aligned}
\end{equation*}

\end{example}

\begin{remark}
The results for the odd powers of $\pi$ in the two examples above agree with those listed by Plouffe\,\cite{Plouffe2}. General formulae for this family of identities for $\pi^{2n+1}$ and $\zeta(2n+1)$, conjectured by Plouffe, were given and proved by Chamberland and Lopatto\,\cite{Chamberland}.
\end{remark}

\begin{proposition}
The results of Theorems \ref{TheoremZeta4k+1p=3} and \ref{TheoremZeta4k+1p=5} can be combined to produce faster converging series for $\pi^{4k+1}$ (than in the example above) in terms of the Lambert series at $q=-e^{-3 \pi },\, e^{-4\pi }, -e^{-5 \pi },\, e^{-6 \pi }$, and $e^{-10 \pi }$. For example,
\begin{equation*}
\begin{aligned}
\pi^5&=-3686634 \, \LambertL_{-e^{-3 \pi }}(-5)+2463048 \, \LambertL_{e^{-4 \pi }}(-5)+402570 \, \LambertL_{-e^{-5 \pi }}(-5)\\
&\hspace{0.15in}+\frac{1843317}{2} \LambertL_{e^{-6 \pi }}(-5)-\frac{201285}{2} \LambertL_{e^{-10 \pi }}(-5).\\
\end{aligned}
\end{equation*}
\end{proposition}

\begin{remark}
Combining Theorems \ref{TheoremZeta4k+1p=5} and \ref{TheoremZeta4k+1root15} gives the fastest converging, but not the cleanest, result for $\pi^{4k+1}$.
\end{remark}

Similarly,

\begin{proposition}
The results of Corollary\,\ref{CorollaryZeta(4k-1)} and Theorem\,\ref{TheoremZeta4k-1root7} can be combined to produce a faster converging series for $\pi^{4k-1}$ (than in the example above) in terms of the Lambert series at $q= e^{-2 \pi },\, -e^{-\sqrt{7} \pi }$, and $e^{-2\sqrt{7} \pi }$. For example,
\begin{equation*}
\begin{aligned}
\pi^3\!=\!\frac{3960}{77-29 \sqrt{7}} \, \LambertL_{e^{-2 \pi }}(-3)-\!\frac{4320}{77-29 \sqrt{7}} \, \LambertL_{-e^{-\sqrt{7} \pi }}(-3)\!+\!\frac{360}{77-29 \sqrt{7}} \, \LambertL_{e^{-2\sqrt{7} \pi }}(-3).
\end{aligned}
\end{equation*}
\end{proposition}
\begin{remark}
Combining the results of Theorems \ref{TheoremZeta4k-1root7} and \ref{TheoremZeta4k-1root15} gives the fastest converging result for $\pi^{4k-1}$.
\end{remark}

\begin{proposition}
By using the methods outlined in Section\,\ref{Sectionzeta(4k+1)} but with $s=-1$ we can calculate the logarithm of the first three primes in terms of the Lambert series\emph{:}
\begin{equation*}
\begin{aligned}
\log{2}&=\frac{2 \pi }{9}-\frac{8}{3}\LambertL_{e^{-2 \pi }}(-1)+\frac{8}{3}\LambertL_{e^{-4 \pi }}(-1),\\
\log{3}&=\frac{19 \pi }{54}-\frac{32}{9}\LambertL_{e^{-2 \pi }}(-1)+\frac{4}{3}\LambertL_{-e^{-3 \pi }}(-1)+\frac{8}{9}\LambertL_{e^{-4 \pi }}(-1)+\frac{4}{3}\LambertL_{e^{-6 \pi }}(-1),\\
\log{5}&=\frac{37 \pi }{72}-\frac{8}{3}\LambertL_{e^{-2 \pi }}(-1)+\frac{2}{3}\LambertL_{e^{-4 \pi }}(-1)+\LambertL_{-e^{-5 \pi }}(-1)+ \LambertL_{e^{-10 \pi }}(-1).\\
\end{aligned}
\end{equation*}
\end{proposition}

\section*{Acknowledgments}

We thank the Van Vleet Physics Professorship and the Mac Armour Physics Fellowship for their support. We thank Simon Plouffe for making the numerically computed results for $\zeta(4k\pm 1)$ available online; we found them to be extremely helpful in formulating our results. We are grateful to Jonathan Sondow and Eric Weisstein for maintaining a comprehensive web page on the Riemann zeta function. The corrections in version~2 were made with the help of Claude (Anthropic), an AI assistant, which checked all formulae numerically against the Riemann zeta function, identified the errors in version~1, and helped update the references.

\newpage

\appendix{Appendix}

We list examples of the first few $\zeta(4k+1)$ generated by Theorem\,\ref{TheoremZeta4k+1p=3}. For compactness only the coefficients are listed respectively for $\pi^{4k+1}$, $\LambertL_{-e^{-3 \pi }}(-4 k-1)$, $\LambertL_{e^{-4 \pi }}(-4 k-1)$, and $\LambertL_{e^{-6 \pi }}(-4 k-1)$.

\begin{equation*}
\begin{aligned}
\zeta(5)&=\left[\frac{682}{201285},-\frac{296}{355},-\frac{488}{355},\frac{74}{355}\right]\\
\zeta(9)&=\left[\frac{5048}{150155775},\frac{2272}{1605},-\frac{5624}{1605},\frac{142}{1605}\right]\\
\zeta(13)&=\left[\frac{21462388}{62314387009875},-\frac{1056896}{2114515},-\frac{3188648}{2114515},\frac{16514}{2114515}\right]\\
\zeta (17)&=\left[\frac{12292037116}{3476479836810605625},\frac{66978304}{95520195},-\frac{258280328}{95520195},\frac{261634}{95520195}\right]\\
\zeta (21)&=\left[\frac{203055579851796692}{5594631411704844933908859375},-\frac{4297066496}{12606788275},-\frac{20920706408}{12606788275},\right.\\
&\left.\hspace{3.77in}\frac{4196354}{12606788275}\right]\\
\zeta (25)&=\left[\frac{91295430825021344}{245007095801727658882798940625},\frac{274844360704}{709832878755},-\frac{1694577218888}{709832878755},\right.\\
&\left.\hspace{3.69in}\frac{67100674}{709832878755}\right]\\
\end{aligned}
\end{equation*}

Examples of the first few $\zeta(4k+1)$ generated by Theorem\,\ref{TheoremZeta4k+1p=5}. The coefficients are listed respectively for $\pi^{4k+1}$, $\LambertL_{e^{-4 \pi }}(-4 k-1)$, $\LambertL_{-e^{-5 \pi }}\!(-4 k-1)$, and $\LambertL_{e^{-10 \pi }}\!(-4 k-1)$. The result for $\zeta(5)$ matches that of Plouffe\,\cite{Plouffe2} when converted to positive $q$ values.

\begin{equation*}
\begin{aligned}
\zeta(5)&=\left[\frac{694}{204813},-\frac{6280}{3251},-\frac{296}{3251},\frac{74}{3251}\right]\\
\zeta(9)&=\left[\frac{6118928}{182032863705},-\frac{3908360}{1945731},\frac{15904}{1945731},\frac{994}{1945731}\right]\\
\zeta(13)&=\left[\frac{4131911428}{11996181573401025},-\frac{2441359240}{1221199811},-\frac{1056896}{1221199811},\frac{16514}{1221199811}\right]\\
\zeta (17)&=\left[\frac{687182059214356}{194362869568557017703375},-\frac{1525878246920}{762905503491},\frac{66978304}{762905503491},\right.\\
&\hspace{3.65in}\left.\frac{261634}{762905503491}\right]\\
\zeta (21)&=\left[\frac{2560199089127112465412}{70537137132904905751999929343125},-\frac{953674355019400}{476839323944771},\right.\\
&\hspace{2.195in}\left.-\frac{4297066496}{476839323944771},\frac{4196354}{476839323944771}\right]\\
\zeta (25)&=\left[\frac{114987316346581920808496}{308598430935986470640664020644801875},-\frac{596046447625404680}{298023086356971651},\right.\\
&\left.\hspace{1.895in}\frac{274844360704}{298023086356971651},\frac{67100674}{298023086356971651}\right]\\
\end{aligned}
\end{equation*}

Examples of the first few $\zeta(4k+1)$ generated by Theorem\,\ref{TheoremZeta4k+1root7}. The coefficients are listed respectively for $\pi^{4k+1}$, $\LambertL_{-e^{-\sqrt{7} \pi }}(-4 k-1)$, and $\LambertL_{e^{-2\sqrt{7} \pi }}(-4 k-1)$. The $\rightarrow$ indicates that the last two coefficients have the same denominator.
\begin{equation*}
\begin{aligned}
\zeta(5)&=\left[\frac{5}{558 \sqrt{7}},-\frac{64}{31},\frac{2}{\hspace{-0.15in}\rightarrow}\right]; \hspace{0.2in}\zeta(9)=\left[\frac{6451}{72571950 \sqrt{7}},-\frac{1088}{543},\frac{2}{\hspace{-0.15in}\rightarrow}\right]\\
\zeta(13)&=\left[\frac{684521}{751529653875 \sqrt{7}},-\frac{16480}{8239},\frac{2}{\hspace{-0.15in}\rightarrow}\right]\\
\zeta (17)&=\left[\frac{7556214529}{808189287857201250 \sqrt{7}},-\frac{261248}{130623},\frac{2}{\hspace{-0.15in}\rightarrow}\right]\\
\zeta (21)&=\left[\frac{11042228011}{115045098786113871375 \sqrt{7}},-\frac{4191904}{2095951},\frac{2}{\hspace{-0.15in}\rightarrow}\right]\\
\zeta (25)&=\left[\frac{93518263081637}{94909028455546692340078125 \sqrt{7}},-\frac{67120832}{33560415},\frac{2}{\hspace{-0.15in}\rightarrow}\right]\\
\zeta (29)&=\left[\frac{62711955554293634}{6199561066546643269478876746875 \sqrt{7}},-\frac{1073792224}{536896111},\frac{2}{\hspace{-0.15in}\rightarrow}\right]\\
\zeta (33)&=\left[\frac{23889742752460078916687}{230049642845129206245609316077442968750 \sqrt{7}},-\frac{17179728128}{8589864063},\frac{2}{\hspace{-0.15in}\rightarrow}\right]\\
\zeta (37)&=\left[\frac{11754448853547401644728513551}{11025844202971263728580541979816097477474609375 \sqrt{7}},\right.\\
&\hspace{3.49in}\left.-\frac{274876959520}{137438479759},\frac{2}{\hspace{-0.1in}\rightarrow}\right]\\
\end{aligned}
\end{equation*}

Examples of Theorem\,\ref{TheoremZeta4k+1root15}. The coefficients are listed respectively for $\pi^{4k+1}$\!, $\mathcal{S}_{e^{-\sqrt{15} \pi }}\!(-4 k-1)$, $\LambertL_{-e^{-\sqrt{15} \pi }}\!(-4 k-1)$, and $\LambertL_{e^{-2\sqrt{15} \pi }}\!(-4 k-1)$.
\begin{equation*}
\begin{aligned}
\zeta(5)&=\left[\frac{5}{378\sqrt{15}},\frac{7}{\sqrt{15}},-\frac{33}{16},\frac{1}{\hspace{-0.1in}\rightarrow}\right]; \hspace{0.2in}\zeta(9)=\left[\frac{19}{145530 \sqrt{15}},\frac{17}{7 \sqrt{15}},-\frac{513}{256},\frac{1}{\hspace{-0.1in}\rightarrow}\right]\\
\zeta(13)&=\left[\frac{5623}{4214184975 \sqrt{15}},\frac{7}{33 \sqrt{15}},-\frac{8193}{4096},\frac{1}{\hspace{-0.1in}\rightarrow}\right]\\
\zeta (17)&=\left[\frac{152161}{11136941565750 \sqrt{15}},-\frac{223}{119 \sqrt{15}},-\frac{131073}{65536},\frac{1}{\hspace{-0.1in}\rightarrow}\right]\\
\zeta (21)&=\left[\frac{2100413011}{15039186678619228125 \sqrt{15}},-\frac{1673}{305 \sqrt{15}},-\frac{2097153}{1048576},\frac{1}{\hspace{-0.1in}\rightarrow}\right]\\
\zeta (25)&=\left[\frac{368670553}{266533834992158608875 \sqrt{15}},-\frac{8143}{231 \sqrt{15}},-\frac{33554433}{16777216},\frac{1}{\hspace{-0.1in}\rightarrow}\right]\\
\zeta (29)&=\left[\frac{276635171660523838}{18471447539635216765490460984375 \sqrt{15}},\frac{30233}{3263 \sqrt{15}},-\frac{536870913}{268435456},\frac{1}{\hspace{-0.1in}\rightarrow}\right]\\
\zeta (33)&=\left[\frac{32756491416931697}{214603915485004924173994292343750 \sqrt{15}},\frac{81343}{26537 \sqrt{15}},-\frac{8589934593}{4294967296},\frac{1}{\hspace{-0.1in}\rightarrow}\right]\\
\zeta (37)&=\left[\frac{4149137440873479901771}{2655425532406550021353118943009341015625 \sqrt{15}},\frac{85673}{133551 \sqrt{15}},\right.\\
&\hspace{3.43in}\left.-\frac{137438953473}{68719476736},\frac{1}{\hspace{-0.1in}\rightarrow}\right]\\
\end{aligned}
\end{equation*}

Examples of the first few $\zeta(4k-1)$ generated by Theorem\,\ref{TheoremZeta4k-1root7}. The coefficients are listed respectively for $\pi^{4k-1}$, $\LambertL_{-e^{-\sqrt{7} \pi }}(-4 k+1)$, and $\LambertL_{e^{-2\sqrt{7} \pi }}(-4 k+1)$.

\begin{equation*}
\begin{aligned}
\zeta(3)&=\left[\frac{29 \sqrt{7}}{1980},-\frac{24}{11},\frac{2}{\hspace{-0.135in}\rightarrow}\right]; \hspace{0.2in} \zeta(7)=\left[\frac{851}{963900 \sqrt{7}},-\frac{240}{119},\frac{2}{\hspace{-0.125in}\rightarrow}\right]\\
\zeta(11)&=\left[\frac{98983}{11006745750 \sqrt{7}},-\frac{3984}{1991},\frac{2}{\hspace{-0.135in}\rightarrow}\right]\\
\zeta (15)&=\left[\frac{120891949}{1310075958262500 \sqrt{7}},-\frac{65712}{32855},\frac{2}{\hspace{-0.135in}\rightarrow}\right]\\
\zeta (19)&=\left[\frac{304799492533}{321754984333646613750 \sqrt{7}},-\frac{1050576}{525287},\frac{2}{\hspace{-0.135in}\rightarrow}\right]\\
\zeta (23)&=\left[\frac{3069248396337203}{315604617827322095616093750 \sqrt{7}},-\frac{16776432}{8388215},\frac{2}{\hspace{-0.135in}\rightarrow}\right]\\
\zeta (27)&=\left[\frac{1080762195172056151}{10825326291702732242830193671875 \sqrt{7}},-\frac{268402704}{134201351},\frac{2}{\hspace{-0.135in}\rightarrow}\right]\\
\zeta (31)&=\left[\frac{28456585682065859550929}{27764709594278970898113257811895312500 \sqrt{7}},-\frac{4294947120}{2147473559},\frac{2}{\hspace{-0.11in}\rightarrow}\right]\\
\end{aligned}
\end{equation*}

Examples of the first few $\zeta(4k-1)$ generated by Theorem\,\ref{TheoremZeta4k-1root15}. The coefficients are listed respectively for $\pi^{4k-1}$\!, $\mathcal{S}_{e^{-\sqrt{15} \pi }}\!(-4 k+1)$, $\LambertL_{-e^{-\sqrt{15} \pi }}\!(-4 k+1)$, and $\LambertL_{e^{-2\sqrt{15} \pi }}\!(-4 k+1)$.

\begin{equation*}
\begin{aligned}
\zeta(3)&=\left[\frac{\sqrt{15}}{100},-\frac{1}{\sqrt{15}},-\frac{9}{4},\frac{1}{\hspace{-0.12in}\rightarrow}\right]; \hspace{0.2in} \zeta(7)=\left[\frac{73}{56700 \sqrt{15}},-\frac{11}{3 \sqrt{15}},-\frac{129}{64},\frac{1}{\hspace{-0.135in}\rightarrow}\right]\\
\zeta(11)&=\left[\frac{82889}{6385128750 \sqrt{15}},-\frac{61}{5 \sqrt{15}},-\frac{2049}{1024},\frac{1}{\hspace{-0.125in}\rightarrow}\right]\\
\zeta (15)&=\left[\frac{78017}{564445381500 \sqrt{15}},\frac{251}{13 \sqrt{15}},-\frac{32769}{16384},\frac{1}{\hspace{-0.125in}\rightarrow}\right]\\
\zeta (19)&=\left[\frac{269130227}{192947512626618750 \sqrt{15}},\frac{781}{171 \sqrt{15}},-\frac{524289}{262144},\frac{1}{\hspace{-0.125in}\rightarrow}\right]\\
\zeta (23)&=\left[\frac{247753871371}{17365083188883060881250 \sqrt{15}},\frac{1451}{989 \sqrt{15}},-\frac{8388609}{4194304},\frac{1}{\hspace{-0.12in}\rightarrow}\right]\\
\zeta (27)&=\left[\frac{362602896277}{2481546577586559529921875 \sqrt{15}},-\frac{2339}{4187 \sqrt{15}},-\frac{134217729}{67108864},\frac{1}{\hspace{-0.135in}\rightarrow}\right]\\
\zeta (31)&=\left[\frac{6262005738704057549}{4186219876052786637474776554687500 \sqrt{15}},-\frac{39589}{13485 \sqrt{15}},\right.\\
&\hspace{3.49in}\left.-\frac{2147483649}{1073741824},\frac{1}{\hspace{-0.12in}\rightarrow}\right]\\
\end{aligned}
\end{equation*}


\begin{thebibliography}{0}

\bibitem{Abramowitz} M. Abramowitz and I. A. Stegun, \textit{Handbook of Mathematical Functions with Formulas, Graphs, and Mathematical Tables} (Dover, New York, 1972).

\bibitem{Apostol} T. M. Apostol, \textit{Modular Functions and Dirichlet Series in Number Theory}, 2nd edn., Graduate Texts in Mathematics, Vol.~41 (Springer, New York, 1990).

\bibitem{BW} S. Banerjee and B. Wilkerson, Asymptotic expansions of Lambert series and related $q$-series, \textit{Int. J. Number Theory} \textbf{13} (2017) 2097--2113.

\bibitem{Berndt} B. C. Berndt, \textit{Ramanujan's Notebooks, Part II} (Springer, New York, 1989).

\bibitem{Berndt1977} B. C. Berndt, Modular transformations and generalizations of several formulae of Ramanujan, \textit{Rocky Mountain J. Math.} \textbf{7} (1977) 147--189.

\bibitem{Broadhurst} D. Broadhurst, Five families of rapidly convergent evaluations of zeta values, arXiv:2401.08997 (2024).

\bibitem{Chamberland} M. Chamberland and P. Lopatto, Formulas for odd zeta values and powers of $\pi$, \textit{J. Integer Seq.} \textbf{14} (2011), Article 11.2.5.

\bibitem{Cohen} H. Cohen, High precision computation of Hardy--Littlewood constants, preprint (2000), https://oeis.org/A221712/a221712.pdf.

\bibitem{Plouffe1} S. Plouffe, Identities inspired from the Ramanujan Notebooks, first series (notes from 1998, revised 2009), arXiv:1101.4826 (2011).

\bibitem{Plouffe2} S. Plouffe, Identities inspired by the Ramanujan Notebooks, second series (first draft 2006), arXiv:1101.6066 (2011).

\bibitem{Sondow} J. Sondow and E. W. Weisstein, Riemann zeta function, \textit{MathWorld---A Wolfram Web Resource}, https://mathworld.wolfram.com/RiemannZetaFunction.html.

\bibitem{Vepstas} L. Vep\v{s}tas, On Plouffe's Ramanujan identities, \textit{Ramanujan J.} \textbf{27} (2012) 387--408.

\end{thebibliography}
\end{document}